\documentclass[11pt,a4paper]{article}
\usepackage[T1]{fontenc}
\usepackage{lmodern}
\usepackage[margin=26mm,top=25mm,bottom=26mm]{geometry}
\usepackage{amsmath,amssymb,amsthm}
\usepackage{microtype}
\usepackage[hidelinks]{hyperref}
\usepackage{authblk}
\newtheorem{theorem}{Theorem}[section]
\newtheorem{lemma}[theorem]{Lemma}
\newtheorem{corollary}[theorem]{Corollary}
\newcommand{\ind}{\mathbf{1}}
\begin{document}
	\title{\bfseries A proof of the strong Erdős–Frankl–Füredi cover-free conjecture}
	
	\author[1]{\bfseries Yanan Hu \thanks{Email: hg@sit.edu.cn}}
	\author[2]{\bfseries Leilei Zhang\thanks{{Corresponding author.~} Email: mathdzhang@163.com}}
	
	\affil[1]{\small School of Science, Shanghai Institute of Technology, Shanghai 201418, China}
	\affil[2]{\small Faculty of Environment and Information Sciences, Yokohama National University, Yokohama 240-8501, Japan}
	\date{}
	\maketitle
	
	\vspace{-2.5em}
	\begin{abstract}
		A family of nonempty sets is $r$-cover-free if no member is contained in the union of at most $r$ other members. Let $T(r)$ denote the least $n$ for which there is an $r$-cover-free family on $n$ points with more than $n$ members. In this paper, we prove that $T(r)\ge(r+1)^2$ for every integer $r\ge1$, confirming the strong form of the Erd\H{o}s--Frankl--F\"uredi cover-free conjecture.  In particular, $T(r)=(r+1)^2$ whenever $r+1$ is a prime power.
	\end{abstract}
	
	\section{Introduction}
	
	For a positive integer $n$, let $[n]=\{1,\ldots,n\}$. A family $\mathcal F$ of nonempty subsets of $[n]$ is \emph{$r$-cover-free} if $A\nsubseteq\bigcup_{B\in\mathcal B}B$ for every $A\in\mathcal F$ and every $\mathcal B\subseteq\mathcal F\setminus\{A\}$ with $|\mathcal B|\le r$. This notion goes back to Kautz and Singleton~\cite{KS}, who studied binary superimposed codes. Later, Erd\H{o}s, Frankl and F\"uredi~\cite{EFF82,EFF85} studied the corresponding extremal problem for families of sets. Define
	$$
	f(n,r)=\max\bigl\{|\mathcal F|:\mathcal F\subseteq 2^{[n]}
	\text{ is }r\text{-cover-free}\bigr\},
	\qquad
	T(r)=\min\{n:f(n,r)>n\}.
	$$
	Since the family of all singletons is $r$-cover-free, we have $f(n,r)\ge n$. Moreover, adjoining a new element and its singleton to an $r$-cover-free family on $[n]$ with more than $n$ members produces such a family on $[n+1]$ with more than $n+1$ members. Consequently, $f(n,r)>n$ for every $n\ge T(r)$.
	
	An earlier lower bound $T(r)\ge\binom{r+2}{2}$ is due to Bassalygo, as reported by Dyachkov and Rykov~\cite{DR};
	see also~\cite[Proposition~3.4]{EFF85}. Using affine planes, Erd\H{o}s, Frankl and F\"uredi~\cite{EFF85} showed that $T(r)\le q^2$ whenever $q>r$ is a prime power, since the $q^2+q$ lines of an affine plane of order $q$ form an $r$-cover-free family on $q^2$ points. In particular, $T(r)\le (r+1)^2$ when $r+1$ is a prime power. They conjectured that $T(r)=(1+o(1))r^2$ and, more strongly, that $T(r)\ge (r+1)^2$ for every $r\ge1$. They also stated without proof that
	$\liminf_{r\to\infty}T(r)/r^2\ge 5/6$ and that the stronger conjecture holds for $r\le3$.

	Huang and Hwang~\cite{HH} proved the stronger conjecture for $r=4$. Together with the affine-plane construction, these results imply $T(1)=4$, $T(2)=9$, $T(3)=16$ and $T(4)=25$. Chen and Hwang~\cite{CH} proved that $T(5)\ge36$. The stronger conjecture has also been verified for several restricted classes. Huang and Hwang~\cite[Section~3]{HH} proved it for families in which any two distinct members intersect in at most one point, and Shangguan and Ge~\cite[Theorem~I.1]{SG} proved it for families in which every member has size $r+1$. Shangguan and Ge~\cite[Corollary~III.5]{SG} also showed that an $r$-cover-free family with more than $n$ members on $n$ points satisfies $n\ge r^2+r+2$ if every member has size at most $\lfloor5r/3\rfloor$. For general families, Shangguan and Ge~\cite[Theorem~I.2]{SG} proved $T(r)\ge\frac{15+\sqrt{33}}{24}r^2$ by counting pairs contained in exactly one member and applying a matching theorem to the remaining pairs. 
	
	More recently, Huang and Liu~\cite[Theorem~1.2]{HL} proved $T(r)\ge r^2+r+2$ for every $r\ge1$, settling the asymptotic
	conjecture. To supplement private-pair counting, they select collections of non-private pairs satisfying a pseudo-private
	condition and control their overlaps using an auxiliary $C_4$-free bipartite graph. In this paper, we prove $T(r)\ge(r+1)^2$ for every $r\ge1$, confirming the strong form stated in~\cite[Conjecture~3.1]{HL}. Using a consequence of Stehl\'ik's theorem~\cite{Stehlik},
	in the dual form recorded in~\cite[Lemma~2.2]{HL}, we establish a weighted covering inequality. We then apply this inequality
	with weights arising from a linear dependence among incidence vectors.

	\begin{theorem}\label{thm:main}
		For every integer $r\ge1$,   $T(r)\ge(r+1)^2$.
	\end{theorem}
	
	Combining Theorem~\ref{thm:main} with the affine-plane construction yields the following corollary.
	
	\begin{corollary}\label{cor:exact}
		If $r+1$ is a prime power, then $T(r)=(r+1)^2$.
	\end{corollary}
	
	The remainder of the paper is organized as follows. In Section~\ref{sec:weighted}, we derive a weighted inequality for covering numbers. In Section~\ref{sec:main}, we apply this inequality to two subfamilies obtained from a linear dependence among the incidence vectors of $\mathcal F$ and prove Theorem~\ref{thm:main}.  
	
	\section{A weighted covering inequality}\label{sec:weighted}
	
	We begin with some definitions and a structural lemma that will be used in the proof.  All families and ground sets are finite. Let $\mathcal H\subseteq2^U$ be a family with union $U$. Its \emph{covering number} $\kappa(\mathcal H)$ is the minimum number of members whose union is $U$. The family $\mathcal H$ is a \emph{down-set} if $S\in\mathcal H$ and $R\subseteq S$ imply $R\in\mathcal H$. For $W\subseteq U$, let $\mathcal H[W]=\{S\in\mathcal H:S\subseteq W\}$. If $\mathcal H$ is a down-set with union $U$, let $G(\mathcal H)$ be the graph on $U$ whose edges are the two-element members of $\mathcal H$. Such a down-set is called \emph{critical} if $\kappa(\mathcal H[U\setminus\{x\}])=\kappa(\mathcal H)-1$ for every $x\in U$.
	
	We use the following consequence of Stehl\'ik's theorem~\cite{Stehlik}, in the dual form recorded by Huang and Liu~\cite[Lemma~2.2]{HL}.
	
	\begin{lemma}[Stehl\'ik]\label{lem:stehlik}
		Let $\mathcal H\subseteq2^U$ be a critical down-set with union $U$. If $G(\mathcal H)$ is connected, then $|U|\ge2\kappa(\mathcal H)-1$.
	\end{lemma}
	
	The following weighted inequality is the main tool in the proof of Theorem~\ref{thm:main}.
	
	\begin{lemma}\label{lem:weighted}
		Let $\mathcal H\subseteq2^U$ be a family with union $U\ne\emptyset$. Assign weights $\lambda_S\ge0$ to its members such that $\sum_{S\in\mathcal H}\lambda_S=1$. For $x\in U$, let $p_x=\sum_{S\in\mathcal H:\,x\in S}\lambda_S$, and suppose that $p_x>0$ for every $x\in U$. Then
		\begin{equation}\label{eq:weighted}
			\kappa(\mathcal H)^2\le\sum_{x\in U}\frac{1}{p_x}.
		\end{equation}
		Empty members and members of weight zero are allowed.
	\end{lemma}
	
	\begin{proof}
		Suppose, to the contrary, that the lemma is false, and choose a counterexample with $|U|$ minimum. Let 
		\begin{equation*}
			m=|U|,\qquad k=\kappa(\mathcal H),\qquad C=\sum_{x\in U}\frac{1}{p_x}. 
		\end{equation*}
		Then $C<k^2$. Since $p_x\le1$ for every $x\in U$, we have $C\ge m\ge1$, and hence $k\ge2$.
		
		We may assume that $\mathcal H$ is a down-set. Indeed, taking its downward closure and assigning weight zero to each new member preserves every $p_x$. It also preserves the covering number: the original covers remain available, and each new member in a cover can be replaced by an original member containing it. In particular, $\{x\}\in\mathcal H$ for every $x\in U$, so $k\le m$.
		
		For a nonempty set $W\subseteq U$, assign to each $R\in\mathcal H[W]$ the weight 
		$$\lambda_R^W=\sum_{S\in\mathcal H:\,S\cap W=R}\lambda_S. $$
		Since $\mathcal H$ is a down-set, every intersection $S\cap W$ belongs to $\mathcal H[W]$. Thus these weights sum to one, including the weight at the empty set, and their sum at each $x\in W$ is $p_x$. By the choice of $U$, inequality~\eqref{eq:weighted} holds for these restricted weights whenever $W$ is a nonempty proper subset of $U$.
		
		We first show that $\mathcal H$ is critical. Fix $x\in U$. Restricting a minimum cover of $U$ and adding $\{x\}$ to a minimum cover of $U\setminus\{x\}$ show that $k-1\le\kappa(\mathcal H[U\setminus\{x\}])\le k$. Since $m\ge k\ge2$, the restricted ground set is nonempty. If its covering number were $k$, its reciprocal sum $C-1/p_x<k^2$ would contradict the choice of $U$. Hence $\kappa(\mathcal H[U\setminus\{x\}])=k-1$, as required.
		
		We distinguish the following two cases according to whether $G(\mathcal H)$ is connected.
		
		\noindent\textbf{Case 1.}  $G(\mathcal H)$ is connected. 
		
		By Lemma~\ref{lem:stehlik}, we have $m\ge2k-1$. Since
		\[
		\sum_{S\in\mathcal H}\lambda_S\sum_{x\in S}\frac1{p_x}
		=\sum_{x\in U}\frac1{p_x}
		\sum_{\substack{S\in\mathcal H\\x\in S}}\lambda_S
		=m,
		\]
		some member $S_0$ of positive weight satisfies $\sum_{x\in S_0}1/p_x\ge m$. Thus $S_0\ne\emptyset$. Also, $S_0\ne U$, since $k\ge2$.
		
		Let $W=U\setminus S_0$ and $k'=\kappa(\mathcal H[W])$. A minimum cover of $W$, together with $S_0$, covers $U$, so $k'\ge k-1$. However,
		\[
		\sum_{x\in W}\frac1{p_x}
		=C-\sum_{x\in S_0}\frac1{p_x}
		\le C-m
		<k^2-(2k-1)
		=(k-1)^2
		\le(k')^2,
		\]
		contrary to the choice of $U$.
		
		\noindent\textbf{Case 2.} $G(\mathcal H)$ is disconnected. 
		
		Let $U_1,\ldots,U_s$ be the vertex sets of its components, where $s\ge2$, and let $k_j=\kappa(\mathcal H[U_j])$. Every nonempty member of $\mathcal H$ lies in one component. Otherwise it would contain points $x$ and $y$ in different components, and downward closure would give $\{x,y\}\in\mathcal H$, a contradiction. Therefore $k=\sum_{j=1}^s k_j$, since each component must be covered separately.
		
		For each $j$, let $\rho_j$ be the total weight of the nonempty members contained in $U_j$. Since $p_x>0$ for every $x\in U_j$, we have $\rho_j>0$. Also, $\sum_{j=1}^s\rho_j\le1$, since the empty set may have positive weight. On $\mathcal H[U_j]$, divide the weight of each nonempty member by $\rho_j$ and assign weight zero to the empty set. These weights sum to one, and their sum at each $x\in U_j$ is $p_x/\rho_j$. Since $U_j$ is a nonempty proper subset of $U$, the choice of $U$ implies $\sum_{x\in U_j}\rho_j/p_x\ge k_j^2$. By the Cauchy--Schwarz inequality,
		$$
		\begin{aligned}
			C=\sum_{j=1}^s\frac1{\rho_j}
			\sum_{x\in U_j}\frac{\rho_j}{p_x}
			\ge\sum_{j=1}^s\frac{k_j^2}{\rho_j}\ge\frac{(\sum_{j=1}^s k_j)^2}{\sum_{j=1}^s\rho_j}
			\ge k^2,
		\end{aligned}
		$$
		a contradiction. This completes the proof.
	\end{proof}
	
	\section{Proof of the main theorem}\label{sec:main}
	
	\begin{proof}[Proof of Theorem~\ref{thm:main}]
		Let $n$ and $r$ be positive integers, and let $\mathcal F$ be an $r$-cover-free family of nonempty subsets of $[n]$
		with $|\mathcal F|>n$. We show that $n\ge(r+1)^2$. 	For $A\subseteq[n]$, let $\ind_A\in\mathbb R^n$ denote its incidence vector, whose $x$-coordinate is $1$ if $x\in A$ and $0$ otherwise. Since $|\mathcal F|>n$, the vectors $\ind_A$, $A\in\mathcal F$, are linearly dependent. These vectors are nonzero and have nonnegative coordinates, so any nontrivial linear relation among them has both positive and negative coefficients. Hence
		\begin{equation}\label{eq:balance}
			\sum_{A\in\mathcal P}a_A\ind_A
			=\sum_{B\in\mathcal N}b_B\ind_B,
		\end{equation}
		where $\mathcal P$ and $\mathcal N$ are disjoint nonempty subfamilies of $\mathcal F$, and all coefficients are positive. Let $a=\sum_{A\in\mathcal P}a_A$ and $b=\sum_{B\in\mathcal N}b_B$. By exchanging the two sides if necessary, we may assume that $a\ge b>0$.
		
		Since all coefficients in~\eqref{eq:balance} are positive, $\mathcal P$ and $\mathcal N$ have the same union, say $U$. For each $x\in U$, let
		\[
		M_x=\sum_{\substack{A\in\mathcal P\\x\in A}}a_A
		=\sum_{\substack{B\in\mathcal N\\x\in B}}b_B>0.
		\]
		
		Fix $A\in\mathcal P$ and let $\mathcal H_A=\{A\cap B:B\in\mathcal N\}$. The family $\mathcal H_A$ covers $A$, since $M_x>0$ for every $x\in A$. We claim that $\kappa(\mathcal H_A)\ge r+1$. Otherwise, at most $r$ members of $\mathcal N$ would cover $A$. Since $\mathcal P\cap\mathcal N=\varnothing$, these members are all distinct from $A$, contrary to the $r$-cover-free condition.
		
		Assign to each $S\in\mathcal H_A$ the weight $\lambda_S=b^{-1}\sum_{B\in\mathcal N:\,A\cap B=S}b_B$. Thus equal intersections have their weights added, and the empty intersection, if present, is included. These weights sum to one, and for each $x\in A$, their sum over the members containing $x$ is $M_x/b>0$. By Lemma~\ref{lem:weighted},
		\[
		(r+1)^2\le\kappa(\mathcal H_A)^2
		\le b\sum_{x\in A}\frac1{M_x}.
		\]
		Multiplying by $a_A$ and summing over $A\in\mathcal P$, we obtain
		\[
		\begin{aligned}
			(r+1)^2a
			\le b\sum_{A\in\mathcal P}a_A\sum_{x\in A}\frac1{M_x}
			=b\sum_{x\in U}\frac{\sum_{A\in\mathcal P:\,x\in A}a_A}{M_x}
			=b|U|\le bn.
		\end{aligned}
		\]
		Since $a\ge b>0$, it follows that $n\ge(r+1)^2a/b\ge(r+1)^2$. This completes the proof.
	\end{proof}

	\begin{proof}[Proof of Corollary~\ref{cor:exact}]
		Let $q=r+1$ and consider the family of all lines of an affine plane of order $q$, as in~\cite{EFF85}. There are $q^2$ points and $q^2+q$ lines. Each line contains $q$ points, and any two distinct lines have at most one common point. Thus any $r=q-1$ other lines cover at most $q-1$ points of a fixed line, so the family is $r$-cover-free. Consequently, $T(r)\le q^2=(r+1)^2$. Together with Theorem~\ref{thm:main}, this proves that $T(r)=(r+1)^2$ whenever $r+1$ is a prime power. In particular, the lower bound is attained for infinitely many values of $r$, so no positive constant can be added to it for all $r$.
	\end{proof}

	\section*{Declaration on the Use of AI}
	
	At an early stage of this work, AI tools were used to assist in  developing  proofs. All results and proofs were written by the authors, who take full responsibility for the content of this work.
	
	\section*{Acknowledgments}
	
	Yanan Hu’s work was supported by the Science and Technology Commission of Shanghai Municipality Grant Number 25ZR1402474. Leilei Zhang’s work was supported by the National Natural Science Foundation of China (Grant Nos. 12601664, 12571365), JSPS KAKENHI (Grant No.~25KF0036), the Hubei Provincial Natural Science Foundation of China (Grant No.~2025AFB309), the China Postdoctoral Science Foundation (Grant No.~2025M773113), the Fundamental Research Funds for the Central Universities, Central China Normal University (Grant No.~CCNU24XJ026).  
	
	\section*{ Declaration}
	
	\noindent\textbf{Conflict~of~interest}
	The authors declare that they have no known competing financial interests or personal relationships that could have appeared to influence the work reported in this paper.
	
	\noindent\textbf{Data~availability}
	No data were used for the research described in this article.
	
\end{document}